\documentclass{amsart}

\usepackage{amsmath,amssymb,amsthm}

\newtheorem{prop}{Proposition}[section]
\newtheorem{thm}[prop]{Theorem}

\newtheorem{cor}[prop]{Corollary}

\theoremstyle{definition}

\newtheorem{defn}[prop]{Definition}

\newtheorem{rem}[prop]{Remark}

\newtheorem*{ack}{Acknowledgement}

\def\co{\colon\thinspace}

\newcommand{\C}{\mathbb C}

\newcommand{\rmd}{\mathrm d}

\newcommand{\R}{\mathbb R}

\newcommand{\ra}{\rightarrow}

\newcommand{\wst}{\omega_{\st}}
\newcommand{\lamst}{\tfrac12(\mathbf{x}\,\rmd\mathbf{y}-\mathbf{y}\,\rmd\mathbf{x})}

\DeclareMathOperator{\eh}{\mathrm{EH}}

\DeclareMathOperator{\st}{\mathrm{st}}

\DeclareMathOperator{\FS}{\mathrm{FS}}

\begin{document}

\author{Kai Zehmisch}
\author{Fabian Ziltener}
\address{Mathematisches Institut, Universit\"at zu K\"oln, Weyertal 86--90, D-50931 K\"oln}
\email{kai.zehmisch@math.uni-koeln.de}
\address{Wiskunde Gebouw, Universiteit Utrecht, Budapestlaan 6, 3584 CD Utrecht}
\email{f.ziltener@uu.nl}

\title[Discontinuous symplectic capacities]{Discontinuous symplectic capacities}

\date{25.11.2013}

\begin{abstract}
  We show that the spherical capacity
  is discontinuous on a smooth family of ellipsoidal shells.
  Moreover,
  we prove that the shell capacity
  is discontinuous on a family of open sets
  with smooth connected boundaries.
\end{abstract}

\thanks{
  The first named author is partially supported by DFG grant ZE 992/1-1.} 

\maketitle

\begin{center}
  \emph{Dedicated to the 90th birthday of Yvonne Choquet-Bruhat}
\end{center}


\section{Introduction\label{intro}}

In \cite{grom85} Gromov proved
that the open ball $B_r=B^{2n}_r$ of radius $r$
embeds symplectically
into the symplectic cylinder $Z_R=B^2_R\times\R^{2n-2}$
if and only if $r\leq R$.
This is the so-called
non-squeezing theorem.
This leads to the Gromov radius
\[
c_B(V,\omega)=\sup\{\pi r^2\,|\,
\text{$\exists$ symplectic embedding
$B_r\hookrightarrow (V,\omega)$}\}\;,
\]
which is a
normalized symplectic capacity.
Following Ekeland and Hofer \cite{ekho89}
we call an assignment
of a real number $c(V,\omega)\in [0,\infty]$
to a symplectic manifold $(V,\omega)$ of dimension $2n$
a 
{\bf normalized symplectic capacity}
provided the following conditions
are satisfied:
\begin{description}
\item[Monotonicity] If there
  exists a symplectic embedding $(V,\omega)\hookrightarrow (V',\omega')$,
  then $c(V,\omega)\leq c(V',\omega')$.
\item[Conformality] For any $a>0$ we have
  $c(V,a\omega)=a\, c(V,\omega)$.
\item[Normalization] $c(B_1)=\pi =c(Z_1)$.
\end{description}
We suppress the standard symplectic form
$\wst=\rmd\mathbf{x}\wedge\rmd\mathbf{y}$
on $\R^{2n}$ in the notation.

A {\bf symplectic embedding} of the
sphere $S^{2n-1}_r=\partial B_r$
of radius $r$
is a symplectic embedding
of a neighbourhood of $S^{2n-1}_r\subset\R^{2n}$.
It is natural to ask
whether a non-squeezing theorem
holds for
symplectic embeddings
$S^{2n-1}_r$ into $Z_R$.
Of course the answer is negative
for the dimension $2n=2$.
If $2n\geq4$
a positive answer was given in \cite{swozil12,geizeh12}.
It was shown
that these embeddings exist precisely if $r<R$.
The spherical non-squeezing theorem
is related to the
{\bf spherical capacity}
\[
 s(V,\omega):=\sup\{\pi r^2\,|\,
\text{$\exists$ symplectic embedding
  $S^{2n-1}_r\hookrightarrow (V,\omega)$}\}\;.
\]
This is a normalized symplectic capacity
for all symplectic manifolds $(V,\omega)$ of dimension $\geq4$,
see \cite{swozil12,geizeh12}.

In \cite[Problem 7]{chls07} the following question was raised:
Consider families of domains with smooth
boundary in a symplectic manifold.
Assume that the boundaries
fit into a smooth isotopy of embeddings.
Are capacities continuous on all
smooth families of domains bounded by smooth hypersurfaces?
The answer is \emph{no}:

\begin{thm}
  \label{mainthm}
  There is a smooth family $U_{\varepsilon}$, $\varepsilon\in(0,1)$,
  of ellipsoidal shells in $\R^{2n}$, $2n\geq 4$,
  such that for the spherical capacity $s$ the function
  $\varepsilon\mapsto s(U_{\varepsilon})$
  is not continuous.
\end{thm}

We prove this result in Section \ref{does}.
Based on the first version of this article
Latschev provided us with the following example
of a discontinuous capacity:
Let $\delta\in(0,1)$.
We define the closed spherical shell
\[
A_{\delta}=
\overline{B_1}\setminus B_{1-\delta}\subset\R^{2n}\;.
\]
Set 
\[
c_{\delta}(V,\omega):=\sup\{\pi r^2\,|\,
\text{$\exists$ symplectic embedding
  $r A_{\delta}\hookrightarrow (V,\omega)$}\}\;.
\]
Observe that $c_{\delta}\leq s$
in dimensions $2n\geq4$.
Therefore, $c_{\delta}$
is a normalized capacity.
We call $c_{\delta}$ the $\delta$-{\bf shell capacity}.
For the open spherical shells
\[
U_{\varepsilon}=
B_1\setminus\overline{B_{1-\varepsilon}}\subset\R^{2n}\;,
\]
$\varepsilon\in(0,1)$,
we have the following:

\begin{thm}
  \label{janko1}
  If $2n\geq4$ then the function
  $\varepsilon\mapsto c_{\delta}(U_{\varepsilon})$
  is discontinuous at $\delta$.
\end{thm}

A similar argument
as in Theorem \ref{janko1}
shows discontinuity of the
$\delta$-shell capacity for a family of subsets of $\C P^n$,
$2n\geq 4$,
with connected boundary.
We provide $\C P^n$ with the Fubini-Study form
$\omega_{\FS}$
normalized such that
the complement of $\C P^{n-1}$ in $\C P^n$
is the symplectic image
of the open unit ball
$B_1$.
We identify both sets
and define
\[
V_{\varepsilon}=\C P^n\setminus\overline{B_{1-\varepsilon}}
\]
for $\varepsilon\in(0,1)$.
Observe that the closure of $V_{\varepsilon}$
is a concave filling of the
sphere $\partial V_{\varepsilon}$
of radius $({1-\varepsilon})$.
The analogue of Theorem \ref{janko1} holds true.

\begin{thm}
  \label{janko2}
  If $2n\geq4$ then the function
  $\varepsilon\mapsto c_{\delta}(V_{\varepsilon})$
  is discontinuous at $\delta$.
\end{thm}

A further variant of the argument
which involves the holomorphic
analysis from \cite{geizeh12}
shows discontinuity on a family
of symplectic fillings of spheres.
Consider the monotone symplectic manifold
$(V,\omega)$
obtained by blowing up
the unit ball in $\R^{2n}$.
For $t>0$ we define $U_t$
to be the corresponding blowup of $B_{1+t}$.
Let $\tau\in(0,\sqrt[2n]{2}-1)$.
Define
\[
c^{\tau}(W,\omega):=\sup\big\{\pi r^2\,|\,
\text{$\exists$ symplectic embedding
  $r\big(\overline{B_{1+\tau}}\setminus B_1\big)\hookrightarrow (W,\omega)$}\big\}\;.
\]

\begin{thm}
  \label{janko3}
  If $2n\geq4$ then the function
  $t\mapsto c^{\tau}(U_t)$
  is discontinuous at $\tau$.
\end{thm}

The proofs of Theorem \ref{janko1},
\ref{janko2}, and \ref{janko3}
are given in Section \ref{jankogenial}.
In Section \ref{acompres} we show that the spherical capacity
and a related contact type embedding capacity
are different.


\section{Discontinuity on ellipsoidal shells\label{does}}

The aim of this section is to prove Theorem \ref{mainthm}.
We consider the open ellipsoid
\[
E :=
E(r_1,\ldots,r_n)=
\left\{\frac{x_1^2+y_1^2}{r_1^2}+\ldots+\frac{x_n^2+y_n^2}{r_n^2}<1\right\}
\]
in $\R^{2n}$
with symplectic half axes $1=r_1\leq\ldots\leq r_n$.
We assume that two of the axes are different
and that $R:=r_n$ satisfies $R<\!\!\sqrt[2n]{2}$.
We define a family of open ellipsoidal shells
\[
U_{\varepsilon}:=(1+\varepsilon)E\setminus\overline{(1-\varepsilon)E}
\]
for $\varepsilon\in(0,1)$.
We claim that the function
$\varepsilon\mapsto s(U_{\varepsilon})$
jumps at
\[
\varepsilon_0 :=\frac{R-1}{R+1}\;.
\]
In the first version of this article
we proved this in dimension $4$
using holomorphic curves.
The argument given here
that is due to the unknown referee
uses Ekeland-Hofer capacities instead
and works in all dimensions $\geq 4$.

\begin{proof}[{\bf Proof of Theorem \ref{mainthm}}]
We consider a symplectic embedding
$\varphi$ of the sphere $S^{2n-1}_r$
into $U_{\varepsilon}$.
We denote the image by $S=\varphi(S^{2n-1}_r)$.
Let $D$ be the bounded component of
$\R^{2n}\setminus S$.

We consider two cases.
Assume that $D\subset U_{\varepsilon}$.
We obtain $|D|<|U_{\varepsilon}|$ for the volume.
Stokes's theorem implies
\[
|D|=\frac{1}{n!}\int_S\lambda\wedge\wst^{n-1}
\]
for a primitive $\lambda$
of $\wst$
and the orientation of $S$ as boundary of $D$.
The integral
is the {\bf helicity} of the
pair $(S,\wst|_{TS})$,
which does not depend
on the choice of the primitive
of $\wst|_{TS}$
provided $n\geq 2$,
see \cite{arkh98} or \cite[p.\ 428]{mcsa98}.
The image of
$\lamst$ under $\varphi$
is a primitive of $\wst$ near $S$.
Because $S$ is the symplectic image
of $S^{2n-1}_r$
the transformation formula yields
$|D|=|B_r|$.
Notice that
the boundary orientation on $S=\partial D$
coincides with the orientation
induced by $\varphi$.
Therefore,
\[
\pi r^2<\sqrt[n]{(1+\varepsilon)^{2n}-(1-\varepsilon)^{2n}}\; \pi R^2\;.
\]

If the domain $D$ is not contained in the shell $U_{\varepsilon}$
we get $(1-\varepsilon)E\subset D$.
A helicity argument as above shows
$1-\varepsilon<r$.
Because $S$ is of restricted contact type
the $n$-th Ekeland-Hofer capacity $c_n^{\eh}$ of $D$
equals a positive integer multiple
of the action of a closed characteristic on $S$,
see \cite[Proposition 2]{ekho90},
i.e.,
there exists a positive integer $m$
such that $c_n^{\eh}(D)=m\pi r^2$.
Because the extrinsic capacity $c_n^{\eh}$
is monotone w.r.t. inclusions,
$D\subset (1+\varepsilon)E$ implies
that $\sqrt{m}\;r\leq (1+\varepsilon)R$.
Combining this with 
$1-\varepsilon<r$ yields
\[\sqrt{m}<\frac{1+\varepsilon}{1-\varepsilon}R.\;\]
The right hand side
is smaller than $\sqrt{2}\;$ if and only if
$\varepsilon<\varepsilon_1$,
where
\[\varepsilon_1=\frac{\sqrt{2}-R}{\sqrt{2}+R}\;.\]
Observe that $\varepsilon_0<\varepsilon_1$
because $R<\sqrt[4]{2}$.
We choose $\varepsilon\in(0,\varepsilon_1]$.
Then $m=1$.
In other words,
the $n$-th Ekeland-Hofer capacity
of $D$ equals $\pi r^2$.
With
$(1-\varepsilon)E\subset D$
this implies
$(1-\varepsilon)R\leq r$.
Moreover,
an application of \cite[Theorem 3.1]{geizeh12}
or \cite{swozil12}
to the inclusion $D\subset (1+\varepsilon)E$
gives $r<1+\varepsilon$.
Combining both 
we get
$(1-\varepsilon)R<1+\varepsilon$.
Therefore,
$\varepsilon_0<\varepsilon$.

To sum up we obtain
\[
s(U_{\varepsilon})\leq
\sqrt[n]{(1+\varepsilon)^{2n}-(1-\varepsilon)^{2n}}\; \pi R^2\;.
\]
provided
$\varepsilon\in (0,\varepsilon_0]$.
The spheres of radius $\frac{2R}{R+1}$
are contained in the closure of
$U_{\varepsilon_0}$
so that
we have the lower bound
\[
\frac{4\pi R^2}{(R+1)^2}\leq s(U_{\varepsilon})
\]
provided
$\varepsilon>\varepsilon_0$.
Consequently, we get
the following estimate
\[
s(U_{\varepsilon_0})
\leq
\sqrt[n]{R^{2n}-1}\;\frac{4\pi R^2}{(R+1)^2}<
s(U_{\varepsilon})
\]
because $R<\sqrt[2n]{2}$.
In other words,
the function
$\varepsilon\mapsto s(U_{\varepsilon})$
is not continuous at $\varepsilon_0$.
This proves Theorem \ref{mainthm}.
\end{proof}

\begin{rem}
  The helicity argument shows
  that the local Liouville vector field
  of a closed connected
  contact type hypersurface
  $(M,\alpha)$
  in an exact symplectic manifold
  that bounds a relative compact domain $D$
  points out of $D$.
  The helicity is taken w.r.t.\ $(M,\rmd\alpha)$
  and coincides with the contact volume of $(M,\alpha)$.
\end{rem}


\section{The shell capacity\label{jankogenial}}

Recall that for $\delta\in(0,1)$
the $\delta$-shell capacity
is 
\[
c_{\delta}(V,\omega):=\sup\{\pi r^2\,|\,
\text{$\exists$ symplectic embedding
  $r A_{\delta}\hookrightarrow (V,\omega)$}\}\;,
\]
where
\[
A_{\delta}=
\overline{B_1}\setminus B_{1-\delta}\subset\R^{2n}\;.
\]

\begin{proof}[{\bf Proof of Theorem \ref{janko1}}]
  Consider a symplectic embedding $\varphi$
  of $r A_{\delta}$ into $U_{\varepsilon}$,
  where
  \[
  U_{\varepsilon}=
  B_1\setminus\overline{B_{1-\varepsilon}}\subset\R^{2n}\;,
  \]
  $\varepsilon\in(0,1)$.
  Set 
  \[
  S_{r(1-\delta)}=\varphi(rS^{2n-1}_{1-\delta})
  \qquad\text{and}\qquad
  S_r=\varphi(S_r^{2n-1})\;.
  \]
  If the bounded component of $\R^{2n}\setminus S_r$
  is contained in $U_{\varepsilon}$
  a helicity argument w.r.t.\ $S_r$
  as in Theorem \ref{mainthm}
  gives $|B_r|<|U_{\varepsilon}|$.
  Hence,
  \[
  r^2<\sqrt[n]{1-(1-\varepsilon)^{2n}}\;.
  \]
  On the other hand
  if the bounded component of
  $\R^{2n}\setminus S_{r(1-\delta)}$
  contains $B_{1-\varepsilon}$
  we obtain
  $|B_{1-\varepsilon}|<|rB_{1-\delta}|$
  by a similar argument.
  Since the bounded component
  of $\R^{2n}\setminus S_r$
  is contained in $B_1$
  we have
  $r^{2n}|B_1|<|B_1|$.
  Hence, $1-\varepsilon<r(1-\delta)$ and $r<1$.
  This yields $\varepsilon>\delta$.
  Because $U_{\delta}$ is contained in $U_{\varepsilon}$
  provided $\delta\leq\varepsilon$
  the alternative case appears precisely if $\varepsilon>\delta$.
  Therefore,
  \[
  c_{\delta}(U_{\varepsilon})\leq\sqrt[n]{1-(1-\delta)^{2n}}\;\pi
  \]
  if $\varepsilon\leq\delta$ and
  $c_{\delta}(U_{\varepsilon})=\pi$
  if $\varepsilon>\delta$.
  This proves Theorem \ref{janko1}.
\end{proof}

\begin{rem}
  If the dimension equals $2n=2$
  the helicity argument from Theorem \ref{mainthm}
  does not apply.
  In fact, the annulus
  $rU_{\delta}=U_{\varepsilon}$ embeds into
  $U_{\varepsilon}\setminus\R^+$ preserving the area,
  see \cite{grsh79,damo90}
  and cf.\ \cite[p.\ 53]{hoze94}.
  Therefore, the function
  $\varepsilon\mapsto c_{\delta}(U_{\varepsilon})$ is continuous.
  The same argument shows that
  $c_{\delta}$ is not normalized.
\end{rem}

\begin{proof}[{\bf Proof of Theorem \ref{janko2}}]
  The proof is based on the following observation:
  Let $S\subset\C P^n$ be a symplectic image of
  the sphere of radius $\varrho$.
  Because $\C P^n$
  has trivial homology in degree $2n-1$
  the complement of $S$
  has two connected components.
  The second de Rham cohomology
  of $\C P^n$ is generated
  by the class of the Fubini-Study form
  and splits corresponding to the components
  of the complement of $S$.
  This follows with the Mayer-Vietoris sequence.
  Denote by $D$ the domain in $\C P^n$
  which bounds $S=\partial D$
  and on which the Fubini-Study form
  has a primitive.
  An application of the helicity argument
  from Theorem \ref{mainthm}
  to $(S,\omega_{\FS}|_{TS})$
  shows that
  the volume of $D$ taken w.r.t.\
  $\frac{1}{n!}\omega_{\FS}^n$
  equals $|B_{\varrho}|$.
  We call $D$ the {\bf interior component}
  of $\C P^n\setminus S$.
  
  If $\varphi$ is a symplectic embedding
  of $rA_{\delta}$ into
  $V_{\varepsilon}=\C P^n\setminus\overline{B_{1-\varepsilon}}$
  then either the interior component of
  $\varphi(S_r^{2n-1})$ is contained
  in  $V_{\varepsilon}$
  or $B_{1-\varepsilon}$
  is contained in
  the interior component of
  $\varphi(rS^{2n-1}_{1-\delta})$.
  In view of this alternative
  the computations made in the proof of
  Theorem \ref{janko1}
  show discontinuity of
  $\varepsilon\mapsto c_{\delta}(V_{\varepsilon})$
  at $\delta$.
\end{proof}

\begin{cor}
  $s(\C P^n)=\pi$.
\end{cor}

For $t>0$ let $U_t$
be the monotone symplectic manifold
obtained by blowing up
the unit ball in $B_{1+t}\subset\R^{2n}$.
Let $\tau\in(0,\sqrt[2n]{2}-1)$
and define
\[
c^{\tau}(W,\omega):=\sup\big\{\pi r^2\,|\,
\text{$\exists$ symplectic embedding
  $r\big(\overline{B_{1+\tau}}\setminus B_1\big)\hookrightarrow (W,\omega)$}\big\}\;.
\]

\begin{proof}[{\bf Proof of Theorem \ref{janko3}}]
  Consider a symplectic embedding
  $\varphi$ of $r\big(\overline{B_{1+\tau}}\setminus B_1\big)$
  into $U_t$.
  Denote the bounded component of
  $V\setminus S_r$
  by $D_r$,
  where $S_r$ is the image $\varphi(S^{2n-1}_r)$.
  The symblols
  $D_{r(1+\tau)}$ and $S_{r(1+\tau)}$
  are understood similarly.
  Notice that the symplectic form
  $\omega$ is exact on $D_r$
  if and only if
  $\omega$ is exact on $D_{r(1+\tau)}$.
  
  If $\omega$ is exact on $D_r$
  we get
  \[
  \pi r^2<\sqrt[n]{(1+t)^{2n}-1}\; \pi\;.
  \]
  In the alternative case
  $\omega$ is exact on $V\setminus D_{r(1+\tau)}$
  because the second de Rham cohomology
  is generated by the class of $\omega$.
  Consider $\bar{U}_t\setminus D_{r(1+\tau)}$,
  which is a symplectic cobordism
  with negative end $S_{r(1+\tau)}$
  and positive end $\partial B_{1+t}$,
  see \cite{mcd91,geizeh12}.
  The minimal action $\pi r^2(1+\tau)^2$
  on $S_{r(1+\tau)}$
  is bounded by $\pi (1+t)^2$,
  see \cite[Theorem 3.1]{geizeh12}.
  Hence, $r(1+\tau)<1+t$.
  
  We claim that $1<r$.
  Arguing by contradiction
  we suppose that $1\geq r$.
  Observe that the second homology
  of $D_r$ is generated by the class
  on which $\omega$ integrates to $\pi$.
  The proof of
  \cite[Theorem 6.4]{geizeh12}
  shows that
  $S_r$ has non-trivial homology
  in degree two
  because the assumption $1\geq r$
  excludes bubbling-off
  of the relevant moduli spaces
  of holomorphic spheres.
  This is a contradiction.
  We get $1<r$.
  Therefore, $\tau<t$.
  
  In other words
  $c^{\tau}(U_t)\geq\pi$ if $t\geq \tau$
  and
  \[
  c^{\tau}(U_t)\leq\sqrt[n]{(1+\tau)^{2n}-1}\; \pi
  \]
  provided that $t<\tau$.
  This proves Theorem \ref{janko3}.
\end{proof}


\section{A comparison result\label{acompres}}

The spherical capacity is a variant
of the regular coisotropic capacity of hypersurfaces,
see \cite{swozil12}.
Consider a closed
hypersurface $M$ in a symplectic manifold $(V,\omega)$
such that all characteristics are closed,
form a smooth fibration over the leaf space with fibre $S^1$,
and are contractible in $V$.
Let $\inf(M)$ denote the least positive symplectic area of
a smooth disc in $V$ with boundary on a closed characteristic of $M$.
Set 
\[
a(V,\omega):=\sup\{\inf(M)\,|\,
\text{$M\subset(V,\omega)$}\}\;,
\]
where the supremum runs over all hypersurfaces as described.
This defines a normalized capacity on all symplectically
aspherical symplectic manifolds
and is called
the {\bf regular coisotropic hypersurface capacity},
see \cite{swozil12}.
The restriction to spheres is denoted by $a_S$
and we get
\[
c_B\leq s\leq a_S\leq a\;.
\]

A related capacity is the
{\bf contact type embedding capacity}
\[
c(V,\omega):=\sup\{\inf(\alpha)\,|\,
\text{$\exists$ contact type embedding
  $(M,\alpha)\hookrightarrow (V,\omega)$}\}\;,
\]
see \cite{geizeh13,geizeh12}.
The supremum is taken over all closed
contact manifolds $(M,\alpha)$ of dimension $2n-1$, 
where $\inf(\alpha)$ is the infimum of all positive periods of closed Reeb orbits
w.r.t.\ the contact form $\alpha$.
A contact type
embedding $j\co (M,\alpha)\hookrightarrow (V,\omega)$
is an embedding
such that there
exists a Liouville vector field $Y$
w.r.t.\ $\omega$ defined near $j(M)$
such that $j^*(i_Y\omega)=\alpha$.
If one restricts
to contact manifolds
which are diffeomorphic to the $(2n-1)$-dimensional
sphere
one obtains a normalized capacity $c_S$.
The capacities $a_S$ and $c_S$
yield a proof
of the spherical non-squeezing theorem
and we have
\[
c_B\leq s\leq c_S\leq c\;.
\]

\begin{defn}
  A symplectic embedding of the boundary of an ellipsoid $E$
  with positive symplectic half axes $r_1\leq\ldots\leq r_n$
  is a symplectic embedding of a neighbourhood of
  \[
  \partial E=
  \left\{
    \frac{x_1^2+y_1^2}{r_1^2}+\ldots+\frac{x_n^2+y_n^2}{r_n^2}=1
  \right\}
  \subset\R^{2n}\;.
  \]
  For symplectic manifolds $(V,\omega)$ of dimension $\geq4$
  and arbitrary ellipsoids $E$
  we call 
  \[
  e(V,\omega):=\sup\{\pi r^2_1\,|\,
  \text{$\exists$ symplectic embedding
    $\partial E\hookrightarrow (V,\omega)$}\}
  \]
  the {\bf ellipsoidal capacity}.
\end{defn}

Notice that
\[
s\leq e\leq c_S\;.
\]
The question which now appears is
whether the capacities $s$ and $e$ coincide.

\begin{thm}
  \label{mainthm2}
  The boundary of an ellipsoid with two
  different symplectic half axes
  has a neighbourhood $U\subset\R^{2n}$
  such that $s(U)<e(U)$.
\end{thm}

\begin{proof}
  Let $E$ be an ellipsoid as in the theorem.
  Set $U_{\varepsilon}=(1+\varepsilon)E\setminus\overline{(1-\varepsilon)E}$.
  We have $\pi r_1^2\leq e(U_{\varepsilon})$.

  On the other hand $s(U_{\varepsilon})\ra 0$.
  Indeed,
  consider a symplectic embedding $\varphi$ of $S^{2n-1}_r$ into $U_{\varepsilon}$.
  Denote by $S=\varphi(S^{2n-1}_r)$ the image.
  If the bounded component $D$ of $\R^{2n}\setminus S$
  is contained in $U_{\varepsilon}$
  the volume of $D$ tends to zero.
  If alternatively $D$ contains $(1-\varepsilon) E$
  we argue as follows:
  Comparing the volume we get a lower bound
  \[
  (1-\varepsilon)^nr_1\cdot\ldots\cdot r_n<r^n
  \]
  invoking the helicity.
  For an upper bound observe that $S\subset (1+\varepsilon)Z_{r_1}$.
  Invoking the spherical non-squeezing theorem we get
  $r<(1+\varepsilon)r_1$,
  see \cite{geizeh12}.
  Combining both inequalities yields
  \[
  \left(\frac{1-\varepsilon}{1+\varepsilon}\right)^n<
  \frac{r_1^n}{r_1\cdot\ldots\cdot r_n}\;.
  \]
  Because the $r_j$ are not all the same
  the right hand side is $<1$.
  Therefore, there exists a positive number $\varepsilon_0$,
  which only depends on the $r_j$,
  such that $\varepsilon>\varepsilon_0$.
  The claim follows now by taking $\varepsilon\leq\varepsilon_0$.  
\end{proof}

\begin{rem}
  Both quantities $s$ and $c_S$
  do not define capacities in dimension $2$.
  Since they satisfy the monotonicity axiom
  they would otherwise measure
  the area of the annuli
  $B_{1+\varepsilon}\setminus\overline{B}_{1-\varepsilon}$
  in $\R^2$,
  see \cite{hoze94}.
  A direct argument can be obtained as follows:
  For $s$ observe
  that $(r,\theta)\mapsto(\sqrt{r^2+a},\theta)$
  maps $S^1=\partial B$ symplectically
  to the circle of radius $\sqrt{1+a}$
  for all $a\in (-1,\infty)$.
  For $c_S$ consider the contact form
  $\tfrac12(r^2+a)d\theta$ on $S^1$.
  Its smallest action equals $(1+a)\pi$.
  
  On the other hand,
  one can measure the greatest minimal action
  of embeddings of restricted contact type
  which have image in a certain open subset.
  This results in an extrinsic normalized capacity
  in dimension $2$ as well,
  cf.\ \cite{geizeh13}.
\end{rem}


\begin{ack}
  We thank the unknown referee and Janko Latschev
  for providing us with the improvements
  as well as Urs Frauenfelder,
  Leonid Polterovich, Felix Schlenk,
  and Jan Swoboda their interest in this work.
  The project was initiated
  during the conference
  \emph{From conservative dynamics to symplectic and contact topology}
  from 30 Jul. 2012 through 3 Aug. 2012 at the Lorentz Center in Leiden.
  We would like to thank the organizers
  Hansj\"org Geiges,
  Viktor Ginzburg,
  Federica Pasquotto,
  Bob Rink, and
  Robert Vandervorst
  for many stimulating discussions.
\end{ack}


\begin{thebibliography}{10}

\bibitem{arkh98}
  {\sc V.\ I.\ Arnol'd, B.\ A.\ Khesin}
  \emph{Topological methods in hydrodynamics},
  Applied Mathematical Sciences {\bf 125},
  Springer, New York, (1998).

\bibitem{chls07}
  {\sc K.\ Cieliebak, H.\ Hofer, J.\ Latschev, F.\ Schlenk},
  Quantitative symplectic geometry,
  in: \emph{Dynamics, ergodic theory, and geometry},
  Math.\ Sci.\ Res.\ Inst.\ Publ.,
  {\bf 54},
  Cambridge Univ.\ Press, Cambridge (2007),
  1--44.
  
\bibitem{damo90}
  {\sc B.\ Dacorogna, J.\ Moser},
  On a partial differential equation involving the {J}acobian determinant,
  \emph{Ann.\ Inst.\ H.\ Poincar\'e Anal.\ Non Lin\'eaire} {\bf 7} (1990),
  1--26.
  
\bibitem{ekho89}
  {\sc I.\ Ekeland, H.\ Hofer},
  Symplectic topology and {H}amiltonian dynamics,
  \emph{Math.\ Z.} {\bf 200} (1989),
  355--378.

\bibitem{ekho90}
  {\sc I.\ Ekeland, H.\ Hofer},
  Symplectic topology and {H}amiltonian dynamics.\ {II},
  \emph{Math.\ Z.} {\bf 203} (1990),
  553--567.

\bibitem{geizeh12}
  {\sc H.\ Geiges, K.\ Zehmisch},
  Symplectic cobordisms and the strong Weinstein conjecture,
  \emph{Math.\ Proc.\ Cambridge Philos.\ Soc.} {\bf 153} (2012),
  261--279.

\bibitem{geizeh13}
  {\sc H.\ Geiges, K.\ Zehmisch},
  How to recognize a 4-ball when you see one,
  \emph{M\"unster J.\ Math.} {\bf 6} (2013),
  525--554.
  
\bibitem{grsh79}
  {\sc R.\ E.\ Greene, K.\ Shiohama},
  Diffeomorphisms and volume-preserving embeddings of noncompact manifolds,
  \emph{Trans.\ Amer.\ Math.\ Soc.} {\bf 255} (1979),
  403--414.
  
\bibitem{grom85}
  {\sc M.\ Gromov},
  Pseudoholomorphic curves in symplectic manifolds,
  \emph{Invent.\ Math.} {\bf 82} (1985),
  307--347.

\bibitem{hoze94}
  {\sc H.\ Hofer, E.\ Zehnder},
  \emph{Symplectic invariants and {H}amiltonian dynamics},
  Birkh\"auser Verlag, Basel, (1994).

\bibitem{mcd91}
  {\sc D.\ McDuff},
  Symplectic manifolds with contact type boundaries,
  \emph{Invent.\ Math.} {\bf 103} (1991),
  651--671.

\bibitem{mcsa98}
  {\sc D.\ McDuff, D.\ Salamon},
  \emph{Introduction to symplectic topology},
  Oxford Mathematical Monographs, Second,
  The Clarendon Press Oxford University Press, New York, (1998),
  1--486.
  
\bibitem{swozil12}
  {\sc J.\ Swoboda, F.\ Ziltener},
  Coisotropic displacement and small subsets of a symplectic manifold,
  \emph{Math.\ Z.} {\bf 271} (2012),
  415--445.
\end{thebibliography}
\end{document}